\documentclass{amsart}
\usepackage[all]{xy}
\usepackage{verbatim}
\usepackage{color}
\usepackage{amsthm}
\usepackage{amssymb}
\usepackage[colorlinks=true]{hyperref}
\usepackage{footmisc}
\numberwithin{equation}{section}

\newtheorem{theorem}[equation]{Theorem}
\newtheorem*{theorem*}{Theorem} \newtheorem{lemma}[equation]{Lemma}

\newtheorem*{conjecture*}{Mamma Conjecture}
\newtheorem*{conjecture1*}{Mamma Conjecture (revisited)}
\newtheorem{proposition}[equation]{Proposition}

\newtheorem*{corollary*}{Corollary}

\theoremstyle{remark}

\newtheorem{example}[equation]{Example}

\newtheorem{notation}[equation]{Notation}

\theoremstyle{remark}
\newtheorem{remark}[equation]{Remark}

\newcommand{\cA}{{\mathcal A}}
\newcommand{\cB}{{\mathcal B}}
\newcommand{\cC}{{\mathcal C}}
\newcommand{\cD}{{\mathcal D}}

\newcommand{\cO}{{\mathcal O}}

\newcommand{\cU}{{\mathcal U}}

\newcommand{\cX}{{\mathcal X}}

\newcommand{\cZ}{{\mathcal Z}}

\newcommand{\Spt}{\mathrm{Spt}}% Spectra

\newcommand{\bbA}{\mathbb{A}}

\newcommand{\bbG}{\mathbb{G}}

\newcommand{\bbP}{\mathbb{P}}

\newcommand{\bbS}{\mathbb{S}}

\newcommand{\bbQ}{\mathbb{Q}}
\newcommand{\bbZ}{\mathbb{Z}}

\DeclareMathOperator{\id}{id}

\DeclareMathOperator{\NMot}{NMot}
\DeclareMathOperator{\NNum}{NNum} % category of noncommutative numerical motives

\DeclareMathOperator{\Num}{Num} % category of numerical motives

\DeclareMathOperator{\Fun}{Fun} % Functor category
\newcommand{\dgcat}{\mathrm{dgcat}} % codimension 
\newcommand{\perf}{\mathrm{perf}}

\newcommand{\dg}{\mathrm{dg}}

\newcommand{\Hom}{\mathrm{Hom}}

\newcommand{\Ho}{\mathrm{Ho}}

\newcommand{\Hmo}{\mathrm{Hmo}}% Morita homotopy theory
\newcommand{\op}{\mathrm{op}}

\newcommand{\too}{\longrightarrow}

\newcommand{\eg}{\textsl{e.g.}}

\let\oldmarginpar\marginpar
\def\marginpar#1{\oldmarginpar{\tiny #1}}

\begin{document}

\title[Noncommutative rigidity]{Noncommutative rigidity}
\author{Gon{\c c}alo~Tabuada}
\address{Gon{\c c}alo Tabuada, Department of Mathematics, MIT, Cambridge, MA 02139, USA}
\email{tabuada@math.mit.edu}
\urladdr{http://math.mit.edu/~tabuada}
\thanks{The author was partially supported by a NSF CAREER Award}

\date{\today}
\abstract{In this article we prove that the numerical Grothendieck group of every smooth proper dg category is invariant under primary field extensions, and also that the mod-$n$ algebraic $K$-theory of every dg category is invariant under extensions of separably closed fields. As a byproduct, we obtain an extension of Suslin's rigidity theorem, as well as of Yagunov-{\O}stv{\ae}r's equivariant rigidity theorem, to singular varieties. Among other applications, we show that base-change along primary field extensions yields a faithfully flat morphism between noncommutative motivic Galois groups. Finally, along the way, we introduce the category of $n$-adic noncommutative mixed motives.}}

\maketitle
\vskip-\baselineskip
\vskip-\baselineskip
%\vskip-\baselineskip
%\tableofcontents

%\bigskip

%\medskip
%-------------------------------------------------------------------------------
\section{Introduction}
%-------------------------------------------------------------------------------
Let $l/k$ be a field extension and $X$ an algebraic $k$-variety. On the one hand, it is (well-)known that when the field extension $l/k$ is primary\footnote{Recall that a field extension $l/k$ is called {\em primary} if the algebraic closure of $k$ in $l$ is purely inseparable over $k$. Whenever $k$ is algebraically closed, every field extension $l/k$ is primary.} and the algebraic $k$-variety $X$ is smooth and proper, base-change induces an isomorphism between the $\bbQ$-vector spaces of algebraic cycles up to numerical equivalence:
\begin{equation}\label{eq:iso-1}
(-)_l\colon  \cZ^\ast(X)_\bbQ/_{\!\!\sim \mathrm{num}} \stackrel{\simeq}{\too} \cZ^\ast(X_l)_\bbQ/_{\!\!\sim \mathrm{num}}\,.
\end{equation}
On the other hand, when $l/k$ is an extension of algebraically closed fields, a remarkable result of Suslin \cite{Suslin} asserts that, for every integer $n\geq 2$ coprime to $\mathrm{char}(k)$, base-change induces an isomorphism in mod-$n$ $G$-theory:
\begin{equation}\label{eq:isos1}
(-)_l\colon  G_\ast(X;\bbZ/n) \stackrel{\simeq}{\too} G_\ast(X_l;\bbZ/n) \,.
\end{equation}
 Among other applications, the isomorphism \eqref{eq:isos1} (with $X=\mathrm{Spec}(k)$) enabled Suslin to describe the torsion of the algebraic $K$-theory of every algebraically closed field of positive characteristic, thus solving a longstanding conjecture of Quillen-Lichtenbaum; consult Suslin's ICM address \cite{Suslin-ICM} for further applications.

%A meticulous inspection shows that the solely ingredient used in the proof of Theorem \ref{thm:Suslin} which is ``sensitive'' to $X$ is the d\'evissage property of $G$-theory. 
The main goal of this article is to establish far-reaching noncommutative generalizations of the above rigidity isomorphisms \eqref{eq:iso-1}-\eqref{eq:isos1}; consult \S\ref{sec:applications} for applications.
\subsection*{Statement of results}
A {\em differential graded (=dg) category} $\cA$, over a base field $k$, is a category enriched over complexes of $k$-vector spaces; see \S\ref{sub:dg}. Every (dg) $k$-algebra $A$ gives naturally rise to a dg category with a single object. Another source of examples is proved by algebraic varieties (or more generally by algebraic stacks) since the category of perfect complexes $\perf(X)$, resp. the bounded derived category of coherent $\cO_X$-modules $\cD^b(\mathrm{coh}(X))$, of every algebraic $k$-variety $X$ admits a canonical dg enhancement $\perf_\dg(X)$, resp. $\cD^b_\dg(\mathrm{coh}(X))$; see \cite[\S4.6]{ICM-Keller}. Following Kontsevich \cite{Miami,finMot,IAS}, a dg category $\cA$ is called {\em smooth} if it is compact as a bimodule over itself and {\em proper} if $\sum_i \mathrm{dim}_k H^i\cA(x,y)<\infty$ for every ordered pair of objects $(x,y)$. Examples include finite dimensional $k$-algebras $A$ of finite global dimension (when $k$ is perfect) as well as the dg categories of perfect complexes $\perf_\dg(X)$ associated to smooth proper algebraic $k$-varieties.

Given a smooth proper dg category $\cA$, recall from \S\ref{sub:numerical} the definition of its numerical Grothendieck group $K_0(\cA)/_{\!\!\sim \mathrm{num}}$. In the same vein, given a dg category $\cA$ and an integer $n \geq 2$, recall from \S\ref{sub:coefficients} the definition of the mod-$n$ algebraic $K$-theory groups $K_\ast(\cA;\bbZ/n)$ as well as of its variants $K_\ast(\cA)^{\wedge}_n$, $K_\ast^{\mathrm{et}}(\cA;\bbZ/n)$, $K_\ast(\cA)/n$, and ${}_n K_\ast(\cA)$. Under these notations, our main result is the following:
\begin{theorem}[Noncommutative rigidity]\label{thm:main}
Given a field extension $l/k$, a dg $k$-linear category $\cA$, and an integer $n \geq 2$, the following holds:
\begin{itemize}
\item[(i)] When the field extension $l/k$ is primary and the dg category $\cA$ is smooth and proper, base-change induces an isomorphism:
\begin{equation}\label{eq:iso-2}
-\otimes_k l \colon K_0(\cA)_\bbQ/_{\!\!\sim \mathrm{num}} \stackrel{\simeq}{\too} K_0(\cA\otimes_k l)_\bbQ/_{\!\!\sim \mathrm{num}}\,.
\end{equation} 
The same holds integrally whenever $k$ is algebraically closed.
\item[(ii)] When $l/k$ is an extension of separably closed fields and $n$ is coprime to $\mathrm{char}(k)$, base-change induces an isomorphism:
\begin{equation}\label{eq:iso-3}
-\otimes_k l \colon K_\ast(\cA;\bbZ/n) \stackrel{\simeq}{\too} K_\ast(\cA\otimes_k l;\bbZ/n)\,.
\end{equation}
Similarly for the variants $K_\ast(-)^\wedge_n$, $K_\ast^{\mathrm{et}}(-;\bbZ/n)$, $K_\ast(-)/n$, and ${}_n K_\ast(-)$.
\end{itemize}
\end{theorem}
%\begin{remark}
%\begin{itemize}
%\item[(i)] Item (i) holds integrally in the particular case where $k$ is algebraically closed; consult Remark \ref{rk:last1}.
%\item[(ii)] Item (ii) also holds for the variants $K_\ast(-)/n$ and ${}_n K_\ast(-)$ of mod-$n$ algebraic $K$-theory; consult Remark \ref{rk:last}. \Goncalo{[Add $n$-adic $K$-theory $K_\ast(-)_n$]}
%\end{itemize}
%\end{remark}
\begin{remark}[Euler pairing]
As explained in \S\ref{sub:numerical}, the numerical Grothendieck group $K_0(\cA)/_{\!\!\sim \mathrm{num}}$ comes equipped with a non-degenerate Euler bilinear pairing $\chi$. Since base-change preserves the Euler bilinear pairing, the above rigidity isomorphism \eqref{eq:iso-2} holds moreover for all those invariants which can be extracted from the pair $(K_0(\cA)/_{\!\!\sim \mathrm{num}}, \chi)$. Among others, these invariants include the Neron-Severi group $\mathrm{NS}(\cA)$ of surface-like dg categories $\cA$ in the sense of Kuznetsov \cite[\S3]{Kuznetsov}.
\end{remark}
To the best of the author's knowledge, Theorem \ref{thm:main} is new in the literature. In what follows, we illustrate its strength via several examples:
\begin{example}[Algebraic cycles up to numerical equivalence]
Let $X$ be a smooth proper algebraic $k$-variety. Thanks to the Hirzebruch-Riemann-Roch theorem, by applying Theorem \ref{thm:main}(i) to the dg category $\cA=\perf_\dg(X)$, we recover the original rigidity isomorphism \eqref{eq:iso-1}. 
%\marginpar{\Goncalo{[When $\mathrm{char}(k)=0$, the numerical Gorthendieck group $K_0(A)/_{\!\!\sim \mathrm{num}}$ is finitely generated (it is always free!). Consequently, the $\bbQ$-vector space $K_0(A)_\bbQ/_{\!\!\sim \mathrm{num}}$ is finite dimensional]}}
\end{example}
\begin{example}[$G$-theory]
Let $X$ be an algebraic $k$-variety and $l/k$ an extension of algebraically closed fields. By applying Theorem \ref{thm:main}(ii) to the dg category $\cA=\cD^b_\dg(\mathrm{coh}(X))$, we recover Suslin's original rigidity isomorphism~\eqref{eq:isos1}.
%$$(-)_l\colon  G_\ast(X;\bbZ/n) \stackrel{\simeq}{\too} G_\ast(X_l;\bbZ/n) \,.$$
%In contrast with Suslin's arguments, our proof does {\em not} uses d\'evissage.\marginpar{\Goncalo{[Proof of the claim that $\cD^b_\dg(\mathrm{coh}(X))_l$ is Morita equivalent to $\cD^b_\dg(\mathrm{coh}(X_l))$: in the affine case run the same argument as in ``quotient singularities'' (using Keller, I know understand Rouquier's claim about generators). In the global case, use Bondal-Vdb, localization and excision (the equivalence of Lemma 3.4.2 (ii) of Sedano restricts to coherents)]}}
\end{example}
\begin{example}[Algebraic $K$-theory]
Let $X$ be an algebraic $k$-variety. By applying Theorem \ref{thm:main}(ii) to the dg category $\cA=\perf_\dg(X)$, we obtain the isomorphism:
\begin{equation}\label{eq:isos3}
(-)_l\colon  K_\ast(X;\bbZ/n) \stackrel{\simeq}{\too} K_\ast(X_l;\bbZ/n) \,.
\end{equation}
It is well-known that $G$-theory agrees with algebraic $K$-theory for nonsingular algebraic varieties. Therefore, in the particular case of an extension $l/k$ of algebraically closed fields, the isomorphism \eqref{eq:isos3} may be understood as the extension of Suslin's rigidity isomorphism \eqref{eq:isos1} to the case of singular algebraic varieties.
\end{example}
\begin{example}[Algebraic stacks]\label{ex:1}
Let $\cX$ be an algebraic $k$-stack. By applying Theorem \ref{thm:main}(ii) to the dg category $\cA=\perf_\dg(\cX)$, we obtain the isomorphism:
\begin{equation}\label{eq:iso-stack}
(-)_l\colon  K_\ast(\cX;\bbZ/n) \stackrel{\simeq}{\too} K_\ast(\cX_l;\bbZ/n) \,.
\end{equation}
Similarly for all the other variants and also for mod-$n$ $G$-theory.
\end{example}
\begin{example}[Equivariant algebraic $K$-theory]\label{ex:2}
Let $G$ be an algebraic group acting on an algebraic $k$-variety $X$. In the particular case where $\cX$ is the global orbifold $[X/G]$, \eqref{eq:iso-stack} reduces to the isomorphism in equivariant algebraic $K$-theory:
\begin{equation}\label{eq:iso-verylast}
(-)_l\colon K^G_\ast(X;\bbZ/n) \stackrel{\simeq}{\too} K^G_\ast(X_l;\bbZ/n) \,.
\end{equation}
When $X$ is smooth and $l/k$ is an extension of algebraically closed fields, this isomorphism was originally established by Yagunov-{\O}stv{\ae}r in \cite{YO}. The rigidity isomorphism \eqref{eq:iso-verylast} holds similarly for all the other variants and also for mod-$n$ $G$-theory.
\end{example}
\begin{example}[Twisted algebraic $K$-theory]\label{ex:3}
Let $X$ be an algebraic $k$-variety and $[\alpha] \in H^2_{\mathrm{et}}(X;\bbG_m)$ a (not necessarily torsion) \'etale cohomology class. By applying Theorem \ref{thm:main}(ii) to the dg category of $\alpha$-twisted perfect complexes $\cA=\perf_\dg(X;\alpha)$, we obtain the isomorphism in twisted algebraic $K$-theory
$$
(-)_l\colon K^\alpha_\ast(X;\bbZ/n) \stackrel{\simeq}{\too} K^{\mathrm{res}(\alpha)}_\ast(X_l;\bbZ/n) \,,
$$
where $\mathrm{res}\colon H^2_{\mathrm{et}}(X;\bbG_m) \to H^2_{\mathrm{et}}(X_l;\bbG_m)$ stands for the restriction homomorphism. Similarly for all the other variants and also for mod-$n$ $G$-theory.
\end{example}
%\begin{remark}
%The above Examples \ref{ex:1}, \ref{ex:2} and \ref{ex:3} hold similarly for $G$-theory. \Goncalo{[Develop the other variants]}
%\end{remark}
%\Goncalo{[Explanation]}
%\begin{corollary}\label{cor:main}
%For every regular field extension $l/k$, with $k$ perfect, the induced base-change functor $-\otimes_k l\colon \NNum(k)_\bbQ \to \NNum(l)_\bbQ$ is fully-faithful. 
%\end{corollary}
%%-------------------------------------------------------------------------------
\section{Applications}\label{sec:applications}
%%-------------------------------------------------------------------------------
%%-------------------------------------------------------------------------------
\subsection*{Noncommutative motivic rigidity}
%%-------------------------------------------------------------------------------
Recall from \cite[\S4.6]{book} the construction of the category of noncommutative numerical motives $\NNum(k)$, and from \cite[\S4]{Hopf} the construction of the triangulated category of noncommutative mixed motives $\mathrm{NMix}(k;\bbZ/n)$ (with $\bbZ/n$-coefficients). In the same vein, recall from \S\ref{sec:adic} the construction of the category of $n$-adic noncommutative mixed motives $\mathrm{NMix}(k)^\wedge_n$.
\begin{theorem}\label{thm:NCmotivic}
\begin{itemize}
\item[(i)] Given a primary field extension $l/k$, the base-change functor $-\otimes_k l\colon \NNum(k)_\bbQ \to \NNum(l)_\bbQ$ is fully-faithful. The same holds integrally whenever $k$ is algebraically closed.  
\item[(ii)] Given an extension $l/k$ of separably closed fields and an integer $n\geq 2$ coprime to $\mathrm{char}(k)$, the following base-change functors are fully-faithful:
$$
-\otimes_kl \colon \mathrm{NMix}(k)^\wedge_n \too \mathrm{NMix}(l)^\wedge_n \quad -\otimes_k l\colon \mathrm{NMix}(k;\bbZ/n) \too \mathrm{NMix}(l;\bbZ/n)\,.
$$
\end{itemize}
\end{theorem}
On the one hand, the commutative counterpart of item (i) was established by Kahn in \cite[Prop.~5.5]{Kahn}; consult Remark \ref{rk:alternative} below for an alternative proof. On the other hand, in the particular case where $l/k$ is an extension of algebraically closed fields, the commutative counterpart of item (ii) was established by R\"ondigs-{\O}stv{\ae}r in \cite[Thm.~1.1]{RO} ($n$-adic case) and by Haesemeyer-Hornbostel in \cite[Thm.~30]{HH} ($\bbZ/n$-coefficients case).
%%-------------------------------------------------------------------------------
\subsection*{Noncommutative motivic Galois groups}
%%-------------------------------------------------------------------------------
Recall from \cite[\S6]{book} the definition of the (conditional) noncommutative motivic Galois group $\mathrm{Gal}(\NNum^\dagger(k)_\bbQ)$. By combining Theorem \ref{thm:NCmotivic}(i) with the Tannakian formalism, we obtain the result:
\begin{theorem}\label{thm:Galois}
Given a primary field extension $l/k$, the induced base-change functor $-\otimes_k l \colon \NNum^\dagger(k)_\bbQ \to \NNum^\dagger(l)_\bbQ$ gives rise to a faithfully flat morphism of affine group schemes $\mathrm{Gal}(\NNum^\dagger(l)_\bbQ) \to \mathrm{Gal}(\NNum^\dagger(k)_\bbQ)$.
\end{theorem}
Intuitively speaking, Theorem \ref{thm:Galois} shows that every ``$\otimes$-symmetry'' of the category of noncommutative numerical $k$-linear motives can be extended to a ``$\otimes$-symmetry'' of the category of noncommutative $l$-linear motives. In the particular case of an extension of algebraically closed fields $l/k$, the commutative counterpart of Theorem \ref{thm:Galois} was established by Deligne-Milne in \cite[Prop.~6.22(b)]{DM}.
%\begin{remark}
%The commutative counterpart of Theorem \ref{thm:Galois} can, alternatively, be obtained by combining the fully-faithfulness of the base-change functor $-\otimes_k l\colon \Num^\dagger(k)_\bbQ \to \Num^\dagger(l)_\bbQ$ (see Remark \ref{rk:diagram}) with the general Tannakian formalism.
%\end{remark}
%%-------------------------------------------------------------------------------
\subsection*{Extra functoriality}
%%-------------------------------------------------------------------------------
Our last application shows that the theory of noncommutative numerical motives is equipped with an extra functoriality:
\begin{theorem}\label{thm:trace}
Given a primary field extension $l/k$, with $\mathrm{char}(k)=0$, the base-change functor $-\otimes_k l\colon \NNum(k)_\bbQ \to \NNum(l)_\bbQ$ admits a left=right adjoint.
\end{theorem}
\begin{proof}
As proved in \cite[Thm.~4.27]{book}, since $\mathrm{char}(k)=0$, the categories $\NNum(k)_\bbQ$ and $\NNum(l)_\bbQ$ are abelian semi-simple. Therefore, the proof follows from the combination of Theorem \ref{thm:NCmotivic}(i) with the general \cite[Prop.~5.3]{Kahn}.
\end{proof}
Roughly speaking, the left=right adjoint functor extracts from each noncommutative numerical $l$-linear motive its largest $k$-linear submotive. The commutative counterpart of Theorem \ref{thm:trace} was established~by~Kahn~in~\cite[Thm.~5.6]{Kahn}.
\begin{remark}\label{rk:false}
Theorem \ref{thm:trace} is {\em false} without the assumption that the field extension $l/k$ is primary. For example, as we explain below, whenever $l=k_{\mathrm{sep}}$ and the field extension $l/k$ is infinite, the functor $-\otimes_k l\colon \NNum(k)_\bbQ \to \NNum(l)_\bbQ$ does {\em not} admits a left adjoint {\em neither} a right adjoint.

By construction, the category of noncommutative numerical motives is equipped with a functor $U(-)_\bbQ\colon \dgcat_{\mathrm{sp}}(k)\to \NNum(k)_\bbQ$ defined on smooth proper dg categories. Let $\mathrm{AM}(k)_\bbQ$ be the idempotent completion of the full subcategory of $\NNum(k)_\bbQ$ consisting of the objects $U(A)_\bbQ$ with $A$ a commutative separable $k$-algebra. As proved in \cite[Prop.~2.3]{Separable}, $\mathrm{AM}(k)_\bbQ$ is equivalent to the classical category of Artin motives. In particular, this latter category is abelian semi-simple. Making use once again of the general \cite[Prop.~5.3]{Kahn}, we hence conclude that the inclusion of categories $\mathrm{AM}(k)_\bbQ \subset \NNum(k)_\bbQ$ admits a left=right adjoint.

Now, let us assume by absurd that the above base-change functor $-\otimes_k l$ admits a left adjoint. This would imply the existence of an Artin motive $N\!\!M$ such that
\begin{equation}\label{eq:equalities}
\mathrm{dim}\Hom_{\mathrm{AM}(k)_\bbQ}(N\!\!M, U(l')_\bbQ) = \mathrm{dim} \Hom_{\NNum(l)_\bbQ}(U(l)_\bbQ, U(l')_\bbQ \otimes_k l) = [l':k]
\end{equation}
for every finite separable field extension $l'/k$. By definition, $N\!\!M$ is a direct summand of $U(l_1 \times \cdots \times l_m)_\bbQ$, where $l_1, \ldots, l_m$ are finite separable field extensions over $k$. Therefore, whenever $l'$ contains $l_1, \ldots, l_m$, we would conclude that the left-hand side of \eqref{eq:equalities} is $\leq [l_1:k] + \cdots + [l_m:k]$. This is a contradiction because, since the extension $l/k$ is infinite, the degree $[l':k]$ can be arbitrarily high. A similar argument shows that the base-change functor $-\otimes_k l$ does not admits a~right~adjoint. 
\end{remark}

%\begin{example}[Abelian varieties]
%Let $\Num(k)_\bbQ$ be the classical category of numerical motives and $\Num(k)_\bbQ/_{\!\! -\otimes \bbQ(1)}$ the associated orbit category with respect to the Tate motive $\bbQ(1)$; see \cite[\S?]{book}. As explained in \cite[\S?]{book}, the category of numerical motives and noncommutative numerical motives are related by a fully faithful functor $\Phi\colon \Num(k)_\bbQ/_{\!\! -\otimes \bbQ(1)} \to \NNum(k)_\bbQ$. Now, let $\mathrm{Ab}(k)_\bbQ$ be the category of abelian $k$-varieties up to isogeny. As explained in \cite[?]{?}, the following composition
%\begin{equation}\label{eq:composition}
%\mathrm{Ab}(k)_\bbQ \stackrel{A\mapsto \mathfrak{h}^1(A)}{\too} \Num(k)_\bbQ \too \Num(k)_\bbQ/_{\!\! -\otimes \bbQ(1)} \stackrel{\Phi}{\too} \NNum(k)_\bbQ
%\end{equation}
%is fully-faithful. Moreover, this composition is compatible with base-change. Therefore, under \eqref{eq:composition}, the value of the adjoint functor $\mathrm{tr}_{l/k}$ at an abelian variety $A$ is given by the classical $l/k$-trace of $A$ (which is isogenous to the $l/k$-image of $A$). For this reason, the adjoint functor $\mathrm{tr}_{l/k}$ may be understood as a generalized $l/k$-trace functor.
% \end{example}

%%-------------------------------------------------------------------------------
\section{Preliminaries}
%%-------------------------------------------------------------------------------
%In order to simplify the exposition, we will make no notational distinction between a dg functor and its image under algebraic $K$-theory (with coefficients). Throughout the note, all tensor products will be taken over the base field $k$.  
%%-------------------------------------------------------------------------------
\subsection{Dg categories}\label{sub:dg}
%%-------------------------------------------------------------------------------
Let $k$ be a commutative ring and $\cC(k)$ the category of complexes of %\marginpar{\Michel{Gon\c calo: I added some extra sections here with things that before were scattered over the body of the manuscript.}}
$k$-modules. A {\em differential
  graded (=dg) category $\cA$} is a category enriched over $\cC(k)$
and a {\em dg functor} $F\colon\cA\to \cB$ is a functor enriched over
$\cC(k)$; consult Keller's ICM survey
\cite{ICM-Keller}. We write $\dgcat(k)$ for the category of dg categories.
%-------------------------------------------------------------------------------
%\subsection{Dg Modules}
%-------------------------------------------------------------------------------

Let $\cA$ be a dg category. The opposite dg category $\cA^\op$ has the
same objects and $\cA^\op(x,y):=\cA(y,x)$. A {\em right dg
  $\cA$-module} is a dg functor $M\colon \cA^\op \to \cC_\dg(k)$ with values
in the dg category $\cC_\dg(k)$ of complexes of $k$-modules. Let
us write $\cC(\cA)$ for the category of right dg
$\cA$-modules. Following \cite[\S3.2]{ICM-Keller}, the derived
category $\cD(\cA)$ of $\cA$ is defined as the localization of
$\cC(\cA)$ with respect to the objectwise quasi-isomorphisms. Let
$\cD_c(\cA)$ be the triangulated subcategory of compact objects.
%-------------------------------------------------------------------------------
%\subsection{Morita equivalences}\label{sub:Morita}
%-------------------------------------------------------------------------------

A dg functor $F\colon\cA\to \cB$ is called a {\em Morita equivalence} if it induces an equivalence on derived categories $\cD(\cA) \simeq
\cD(\cB)$; see \cite[\S4.6]{ICM-Keller}. As explained in
\cite[\S1.6]{book}, the category $\dgcat(k)$ admits a Quillen model
structure whose weak equivalences are the Morita equivalences. Let us
denote by $\Hmo(k)$ the associated homotopy category.
%-------------------------------------------------------------------------------
%\subsection{Tensor product}
%-----------------------------------------------------------------------------equivalence--

The {\em tensor product $\cA\otimes\cB$} of dg categories is defined
as follows: the set of objects is the cartesian product and
$(\cA\otimes\cB)((x,w),(y,z)):= \cA(x,y) \otimes\cB(w,z)$. As
explained in \cite[\S2.3]{ICM-Keller}, this construction gives rise to
a symmetric monoidal structure $-\otimes -$ on $\dgcat(k)$, which descends $-\otimes^{\bf L}-$ to the homotopy category
$\Hmo(k)$. %After deriving it $-\otimes-$, this symmetric monoidal structure descends to $\Hmo(k)$; consult \cite[\S4.3]{ICM-Keller} for details.
\subsection{Numerical Grothendieck group}\label{sub:numerical}
%%-------------------------------------------------------------------------------
Let $k$ be a field. Given a proper dg $k$-linear category $\cA$, its Grothendieck group $K_0(\cA):=K_0(\cD_c(\cA))$ comes equipped with the Euler bilinear pairing $\chi \colon  K_0(\cA) \times K_0(\cA) \to \bbZ$ defined as follows:
\begin{equation}\label{eq:pairing}
([M],[N]) \mapsto \sum_j (-1)^j \mathrm{dim}_k \Hom_{\cD_c(\cA)}(M,N[j])\,.
\end{equation}
This bilinear pairing is, in general, not symmetric neither skew-symmetric. Nevertheless, when $\cA$ is moreover smooth, the associated left and right kernels of $\chi$ agree; see \cite[Prop.~4.24]{book}. Consequently, under these assumptions on $\cA$, we have a well-defined {\em numerical Grothendieck group} $K_0(\cA)/_{\!\!\sim \mathrm{num}}:=K_0(\cA)/\mathrm{Ker}(\chi)$; similarly, consider the $\bbQ$-vector space $K_0(\cA)_\bbQ/_{\!\!\sim \mathrm{num}}$. Note that $K_0(\cA)/_{\!\!\sim \mathrm{num}}$ is torsion-free and that $\chi$ induces a {\em non-degenerate} bilinear pairing on this latter group. When $\mathrm{char}(k)=0$, this latter group is known to be finitely generated;~see~\cite[Thm.~1.2]{Separable}.
\begin{remark}[Generalization]\label{rk:generalization}
Let $k$ be a connected\footnote{Recall that a commutative ring $k$ is called {\em connected} if $\mathrm{Spec}(k)$ is a connected topological space or, equivalently, if $k$ does not contains non-trivial idempotent elements.} commutative ring (\eg\ a field) and $\cA$ a dg $k$-linear category such that $\sum_i \mathrm{rank}_k H^i\cA(x,y)<~\infty$ for every ordered pair of objects $(x,y)$. By replacing $\mathrm{dim}_k\Hom_{\cD_c(\cA)}(M,N[j])$ with $\mathrm{rank}_k\Hom_{\cD_c(\cA)}(M,N[j])$ in the above definition \eqref{eq:pairing}, we obtain a natural generalization of the Euler bilinear pairing $\chi\colon K_0(\cA) \times K_0(\cA) \to \bbZ$.
\end{remark}
%%-------------------------------------------------------------------------------
\subsection{Mod-$n$ algebraic $K$-theory}\label{sub:coefficients}
%%-------------------------------------------------------------------------------
Recall from \cite[\S2.2.4]{book} the construction of nonconnective algebraic $K$-theory $K\colon \dgcat(k) \to \Ho(\Spt)$, with values in the homotopy category of (symmetric) spectra. Given an integer $n\geq 2$, consider the triangle $\bbS \stackrel{n\cdot \id}{\to} \bbS \to \bbS/n \to \Sigma \bbS$, where $\bbS$ stands for the sphere spectrum. Following Browder \cite{Browder} (and Karoubi), mod-$n$ algebraic $K$-theory~is~defined~as~follows:
\begin{eqnarray}\label{eq:K-coefficients}
K(-;\bbZ/n)\colon \dgcat(k) \too \Ho(\Spt) && \cA \mapsto K(\cA) \wedge \bbS/n\,.
\end{eqnarray}
In addition to mod-$n$ algebraic $K$-theory, we can also consider $n$-adic algebraic $K$-theory $K(\cA)^\wedge_n:=\mathrm{holim}_\nu K(\cA; \bbZ/n^\nu)$, mod-$n$ \'etale $K$-theory $K^{\mathrm{et}}(\cA;\bbZ/n):=L_{K(1)}K(\cA;\bbZ/n)$ (see \cite[\S2.2.6]{book}), and the groups $K_\ast(\cA)/n$ and ${}_n K_\ast(\cA)$.
%-------------------------------------------------------------------------------
\section{Proof of Theorem \ref{thm:main}}
%-------------------------------------------------------------------------------
%-------------------------------------------------------------------------------
\subsection*{Proof of item (i)}
%-------------------------------------------------------------------------------
We start by describing the behavior of the numerical Grothendieck group with respect to some field extensions:
\begin{lemma}\label{lem:injectivity}
Given a field extension $l/k$ and a smooth proper dg $k$-linear category $\cA$, we have the following commutative diagram:
$$
\xymatrix{
K_0(\cA\otimes_kl)_\bbQ \times K_0(\cA\otimes_kl)_\bbQ \ar[r]^-\chi & \bbQ \\
K_0(\cA)_\bbQ \times K_0(\cA)_\bbQ \ar[r]_-{\chi} \ar[u]^-{-\otimes_kl} & \bbQ \ar@{=}[u]\,.
}
$$
%Consequently, the base-change homomorphism \eqref{eq:induced-22} is injective.
\end{lemma}
\begin{proof}
Given right dg $\cA$-modules $M, N \in \cD_c(\cA)$, we have the following natural isomorphisms of $l$-vector spaces
\begin{eqnarray*}
\Hom_{\cD_c(\cA)}(M,N[j]) \otimes_kl\simeq\Hom_{\cD_c(\cA\otimes_k l)}(M\otimes_kl, (N \otimes_kl)[j]) && j \in \bbZ\,.
\end{eqnarray*}
This implies that $\chi([M\otimes_k l], [N \otimes_k l])=\chi([M],[N])$. Consequently, the proof follows from the fact that the $\bbQ$-vector space $K_0(\cA)_\bbQ$ is generated by the elements $a_1 [M_1] + \cdots + a_n [M_n]$ with $a_1, \ldots, a_n \in \bbQ$ and $M_1, \ldots, M_n \in \cD_c(\cA)$.
\end{proof}
\begin{remark}[Generalization]\label{rk:generalization1}
Let $k\to l$ be an homomorphism between connected commutative rings. Given a dg $k$-linear category $\cA$ such that $\sum_i \mathrm{rank}_k H^i\cA(x,y)<\infty$ for every ordered pair $(x,y)$ (see Remark \ref{rk:generalization}), a proof similar to the one of Lemma \ref{lem:injectivity} yields the following commutative diagram:
\begin{equation}\label{eq:square-key}
\xymatrix{
K_0(\cA\otimes^{\bf L}_kl)_\bbQ \times K_0(\cA\otimes^{\bf L}_kl)_\bbQ \ar[r]^-\chi & \bbQ \\
K_0(\cA)_\bbQ \times K_0(\cA)_\bbQ \ar[r]_-{\chi} \ar[u]^-{-\otimes^{\bf L}_kl} & \bbQ \ar@{=}[u]\,.
}
\end{equation}
\end{remark}
\begin{lemma}\label{lem:induced}
Let $l/k$ be a field extension and $\cA$ a smooth proper dg $k$-linear category. Whenever the field extension $l/k$ is algebraic or the field $k$ is algebraically closed, base-change induces an injective homomorphism:
\begin{equation}\label{eq:induced-22}
-\otimes_k l \colon K_0(\cA)_\bbQ/_{\!\!\sim \mathrm{num}} \too K_0(\cA\otimes_k l)_\bbQ/_{\!\!\sim \mathrm{num}}\,.
\end{equation}
\end{lemma}
\begin{proof}
The $\bbQ$-vector space $K_0(\cA)_\bbQ$, resp. $K_0(\cA\otimes_k l)_\bbQ$, is generated by the elements $a_1[M_1]+ \cdots + a_n[M_n]$, resp. $b_1[N_1]+ \cdots + b_m[N_m]$, with $a_1, \ldots, a_n \in \bbQ$ and $M_1, \ldots, M_n \in \cD_c(\cA)$, resp. $b_1, \ldots, b_m \in \bbQ$ and $N_1, \ldots, N_m \in \cD_c(\cA\otimes_k l)$. Therefore, given a right dg $\cA$-module $M \in \cD_c(\cA)$, such that $[M]\in \mathrm{Ker}(\chi)$, and a right dg $(\cA\otimes_k l)$-module $N \in \cD_c(\cA\otimes_k l)$, it suffices to show that $\chi([M\otimes_k l], [N])=0$; the injectivity of \eqref{eq:induced-22} follows automatically from the above Lemma \ref{lem:injectivity}.

Let us assume first that $l/k$ is a finite (algebraic) field extension of degree $d$. In this case, we have the following adjunction of categories:
$$
\xymatrix{
\cD_c(\cA\otimes_k l) \ar@<1ex>[d]^-{\mathrm{res}}\\
\cD(\cA) \ar@<1ex>[u]^-{-\otimes_k l}\,.
}
$$
This yields adjunction isomorphisms
\begin{eqnarray*}
\Hom_{\cD_c(\cA\otimes_k l)}(M\otimes_k l, N[j]) \simeq \Hom_{\cD_c(\cA)}(M,\mathrm{res}(N)[j])& j \in \bbZ
\end{eqnarray*}
and hence the following equality:
\begin{equation}\label{eq:equality-key}
\mathrm{dim}_l \Hom_{\cD_c(\cA\otimes_k l)}(M\otimes_k l, N[j]) = d\cdot \mathrm{dim}_k \Hom_{\cD_c(\cA)}(M, \mathrm{res}(N)[j])\,.
\end{equation}
Consequently, we conclude that $\chi([M\otimes_k l], [N])= d\cdot \chi([M], [\mathrm{res}(N)])=0$.

Let us now assume that $l/k$ is an infinite algebraic field extension. In this case, $l$ identifies with the colimit of the filtrant diagram $\{l_i\}_{i \in I}$ of all those intermediate field extensions $l/l_i/k$ which are finite over $k$. This leads to an equivalence $\mathrm{colim}_{i \in I} \cD_c(\cA\otimes_{k}l_i) \simeq \cD_c(\cA\otimes_k l)$. Consequently, there exists an index $i_o \in I$ and a right dg $(\cA\otimes_k l_{i_o})$-module $N_{i_o}\in \cD_c(\cA\otimes_k l_{i_o})$ such that $N_{i_o} \otimes_{l_{i_o}}l \simeq N$. Making use of Lemma \ref{lem:injectivity}, 
%the natural isomorphisms of $l$-vector spaces
%$$\Hom_{\cD_c(\cA\otimes_k l_{i_o})}(M\otimes_k l_{i_o}, N_{i_o}[j])\otimes_{l_{i_o}}l\simeq \Hom_{\cD_c(\cA\otimes_k l)}(M\otimes_k l, N[j]) \quad j \in \bbZ\,,$$
 we hence obtain the equality $\chi([M\otimes_k l], [N])= \chi([M\otimes_k l_{i_o}], [N_{i_o}])$. The proof follows now from the equality $\chi([M\otimes_k l_{i_o}], [N_{i_o}])= d_{i_o} \cdot \chi([M], [\mathrm{res}(N_{i_o})])=0$, where $d_{i_o}$ stands for the degree of the finite field extension $l_{i_o}/k$.
 
Finally, let us assume that $k$ is algebraically closed.  Note that $l$ identifies with the colimit of the filtrant diagram $\{k_i\}_{i \in I}$ of all those finitely generated $k$-algebras $k_i$ which are contained in $l$; note that all these $k$-algebras $k_i$ are connected because they are contained in the field $l$. This leads to an equivalence of categories $\mathrm{colim}_{i \in I} \cD_c(\cA\otimes_{k}k_i) \simeq \cD_c(\cA\otimes_k l)$. Consequently, there exists an index $i_o \in I$ and a right dg $(\cA\otimes_k k_{i_o})$-module $N_{i_o}\in \cD_c(\cA\otimes_k k_{i_o})$ such that $N_{i_o} \otimes^{\bf L}_{k_{i_o}}l \simeq N$. Making use of Remark \ref{rk:generalization1}, we hence obtain the equality $\chi([M\otimes_k l], [N])= \chi([M\otimes_k k_{i_o}], [N_{i_o}])$. Since $k$ is algebraically closed and the $k$-algebra $k_{i_o}$ is finitely generated, the Hilbert's nullstellensatz theorem implies that the $k$-scheme $\mathrm{Spec}(k_{i_o})$ admits a rational point $p\colon \mathrm{Spec}(k) \to \mathrm{Spec}(k_{i_o})$. Hence, we can consider the Grothendieck class $[N_{i_o} \otimes^{\bf L}_{k_{i_o}} k] \in K_0(\cA)_\bbQ$. Using the fact that the composition $k \to k_{i_o} \stackrel{p}{\to} k$ is equal to the identity, we then conclude from Remark \ref{rk:generalization1} that $\chi([M\otimes_k k_{i_o}], [N_{i_o}])= \chi([M], [N_{i_o}\otimes^{\bf L}_{k_{i_o}}k])=0$. This finishes the proof.
\end{proof}

We now describe the behavior of the numerical Grothendieck group with respect to Galois field extensions and purely inseparable field extensions. 
\begin{notation}
Given a dg $k$-linear category $\cA$, let us denote by 
$$-\bullet - \colon K_0(\cA)_\bbQ \times K_0(k)_\bbQ \too K_0(\cA)_\bbQ$$
the bilinear pairing associated to the canonical (right) action of $\cD_c(k)$ on $\cD_c(\cA)$.
\end{notation}
\begin{proposition}[Galois]\label{prop:Galois}
Given a Galois field extension $l/k$ and a smooth proper dg $k$-linear category $\cA$, we have an induced isomorphism:
\begin{equation}\label{eq:induced-Galois}
-\otimes_kl\colon K_0(\cA)_\bbQ/_{\!\!\sim \mathrm{num}} \stackrel{\simeq}{\too} (K_0(\cA\otimes_k l)_\bbQ/_{\!\!\sim \mathrm{num}})^{\mathrm{Gal}(l/k)}\,.
\end{equation}
\end{proposition}
\begin{proof}
Let us assume first that the field extension $l/k$ is finite of degree $d$. In this case we have the following adjunctions of categories:
$$
\xymatrix{
\cD_c(\cA\otimes_k l) \ar@<1ex>[d]^-{\mathrm{res}} && \cD_c(l)  \ar@<1ex>[d]^-{\mathrm{res}}\\
\cD_c(\cA) \ar@<1ex>[u]^-{-\otimes_k l} && \cD_c(k) \ar@<1ex>[u]^-{-\otimes_k l}\,.
}
$$
Note that the equality \eqref{eq:equality-key} in the proof of Lemma \ref{lem:induced} implies not only that the homomorphism $-\otimes_kl \colon K_0(\cA)_\bbQ \to K_0(\cA\otimes_k l)_\bbQ$ preserves the subspaces $\mathrm{Ker}(\chi)$, but also that the homomorphism $\mathrm{res} \colon K_0(\cA\otimes_k l)_\bbQ \to K_0(\cA)_\bbQ$ preserves the subspaces $\mathrm{Ker}(\chi)$. Hence, we can consider the following two homomorphisms:
\begin{eqnarray*}
K_0(\cA)_\bbQ/_{\!\!\sim \mathrm{num}} \stackrel{-\otimes_k l}{\too} K_0(\cA\otimes_k l)_\bbQ/_{\!\!\sim \mathrm{num}} && K_0(\cA\otimes_k l)_\bbQ/_{\!\!\sim \mathrm{num}} \stackrel{\mathrm{res}}{\too} K_0(\cA)_\bbQ/_{\!\!\sim \mathrm{num}}\,.
\end{eqnarray*}
Clearly, the composition $\mathrm{res}\circ (-\otimes_k l)$ is equal to $d\cdot \id$. Given an element $\sigma \in G:=\mathrm{Gal}(l/k)$, let us write $\sigma(-)$ for the associated automorphism of $K_0(\cA\otimes_k l)_\bbQ/_{\!\!\sim \mathrm{num}}$. Under these notations, the following equalities 
$$ \mathrm{res}([N])\otimes_k l = [N] \bullet (\mathrm{res}([l])\otimes_k l)\stackrel{\mathrm{(a)}}{=} [N] \bullet (\sum_{\sigma \in G}\sigma([l]))= \sum_{\sigma \in G}\sigma([N])$$
hold for every $N \in \cD_c(\cA\otimes_k l)$; the equality (a) follows from the fact that $l/k$ is Galois. Consequently, since the $\bbQ$-vector space $K_0(\cA\otimes_k l)_\bbQ$ is generated by the elements $b_1[N_1]+ \cdots + b_m [N_m]$, with $b_1, \ldots, b_m \in \bbQ$ and $N_1, \ldots, N_m \in \cD_c(\cA\otimes_k l)$, we conclude that the composition $\mathrm{res}(-)\otimes_k l$ is equal to $\sum_{\sigma \in G}\sigma(-)$. The proof follows now from the fact that the $\bbQ$-vector space $(K_0(\cA\otimes_k l)_\bbQ/_{\!\!\sim \mathrm{num}})^{G}$ agrees with the image of the idempotent endomorphism $\frac{1}{d} \sum_{\sigma \in G}\sigma(-)$ of $K_0(\cA\otimes_k l)_\bbQ/_{\!\!\sim \mathrm{num}}$.

Let us now assume that the field extension $l/k$ is infinite. In this case, $l$ identifies with the colimit of the filtrant diagram $\{l_i\}_{i \in I}$ of all those intermediate field extensions $l/l_i/k$ which are finite and Galois over $k$. This leads to an equivalence of categories $\mathrm{colim}_{i \in I} \cD_c(\cA\otimes_k l_i) \simeq \cD_c(\cA\otimes_k l)$ and hence to an isomorphism $\mathrm{colim}_{i \in I} K_0(\cA\otimes_k l_i)_\bbQ \simeq K_0(\cA\otimes_k l)_\bbQ$. Thanks to Lemma \ref{lem:induced}, note that we have a similar isomorphism $\mathrm{colim}_{i \in I} K_0(\cA\otimes_k l_i)_\bbQ/_{\!\!\sim \mathrm{num}} \simeq K_0(\cA\otimes_k l)_\bbQ/_{\!\!\sim \mathrm{num}}$. Consequently, the proof follows from the following natural isomorphisms
\begin{eqnarray}
(K_0(\cA\otimes_k l)_\bbQ/_{\!\!\sim \mathrm{num}})^{\mathrm{Gal}(l/k)} & \simeq & (\mathrm{colim}_{i \in I} K_0(\cA\otimes_k l_i)_\bbQ/_{\!\!\sim \mathrm{num}})^{\mathrm{Gal}(l/k)} \nonumber \\
& \simeq & \mathrm{colim}_{i \in I} (K_0(\cA\otimes_k l_i)_\bbQ/_{\!\!\sim \mathrm{num}})^{\mathrm{Gal}(l/k)} \nonumber \\
& \simeq & \mathrm{colim}_{i \in I} (K_0(\cA\otimes_k l_i)_\bbQ/_{\!\!\sim \mathrm{num}})^{\mathrm{Gal}(l_i/k)} \label{eq:star11}\\
& \simeq & \mathrm{colim}_{i \in I} K_0(\cA)_\bbQ/_{\!\!\sim \mathrm{num}} \nonumber\\
& \simeq & K_0(\cA)_\bbQ/_{\!\!\sim \mathrm{num}}\,, \nonumber
\end{eqnarray}
where in \eqref{eq:star11} we are (implicitly) using the surjection $\mathrm{Gal}(l/k) \twoheadrightarrow \mathrm{Gal}(l_i/k)$.
\end{proof}
\begin{proposition}[Purely inseparable]\label{prop:inseparable}
Given a purely inseparable field extension $l/k$ and a smooth proper dg $k$-linear category $\cA$, we have an induced isomorphism:
\begin{equation}\label{eq:inseparable}
-\otimes_k l \colon K_0(\cA)_\bbQ/_{\!\!\sim \mathrm{num}} \stackrel{\simeq}{\too} K_0(\cA\otimes_k l)_\bbQ/_{\!\!\sim \mathrm{num}}\,.
\end{equation}
\end{proposition}
\begin{proof}
Let us assume first that the field extension $l/k$ is finite of degree $d$. Similarly to the proof of Proposition \ref{prop:Galois}, the composition $\mathrm{res} \circ (-\otimes_k l)$ is equal to $d\cdot \id$. On the other hand, the following equalities
$$ \mathrm{res}([N]) \otimes_k l = [N] \bullet (\mathrm{res}([l])\otimes_k l)\stackrel{\mathrm{(a)}}{=} [N] \bullet (d \cdot [l]) = d\cdot [N]$$
hold for every $N \in \cD_c(\cA\otimes_kl)$; the equality (a) follows from \cite[\S7 Prop.~4.8]{Quillen}. Consequently, since the $\bbQ$-vector space $K_0(\cA\otimes_k l)_\bbQ$ is generated by the the elements $b_1[N_1]+ \cdots + b_m [N_m]$, with $b_1, \ldots, b_m \in \bbQ$ and $N_1, \ldots, N_m \in \cD_c(\cA\otimes_k l)$, we conclude that the composition $\mathrm{res}(-)\otimes_k l$ is also equal to $d\cdot \id$. This implies that the above induced homomorphism \eqref{eq:inseparable} is invertible.

Let us now assume that the field extension $l/k$ is infinite. In this case, $l$ identifies with the colimit of the filtrant diagram $\{l_i\}_{i \in I}$ of all those intermediate field extensions $l/l_i/k$ which are finite and purely inseparable over $k$.  This leads to an equivalence of categories $\mathrm{colim}_{i \in I} \cD_c(\cA\otimes_k l_i)\simeq \cD_c(\cA\otimes_k l)$ and hence to an isomorphism $\mathrm{colim}_{i \in I} K_0(\cA\otimes_k l_i)_\bbQ/_{\!\!\sim \mathrm{num}}\simeq K_0(\cA\otimes_k l)_\bbQ/_{\!\!\sim \mathrm{num}}$. Consequently, the proof follows from the preceding finite-dimensional case.
\end{proof}

We now have all the ingredients necessary for the proof of item (i). Let us assume first that $l/k$ is a field extension with $k$ algebraically closed. The injectivity of \eqref{eq:iso-2} follows automatically from Lemma \ref{lem:induced}. In order to prove the surjectivity of \eqref{eq:iso-2}, note first that $l$ identifies with the colimit of the filtrant diagram $\{k_i\}_{i \in I}$ of all those finitely generated $k$-algebras $k_i$ which are contained in $l$; all these $k$-algebras $k_i$ are connected because they are contained in the field $l$. This leads to an equivalence of categories $\mathrm{colim}_{i \in I} \cD_c(\cA\otimes_k k_i)\simeq \cD_c(\cA\otimes_kl)$ and hence to an isomorphism $\mathrm{colim}_{i \in I} K_0(\cA\otimes_k k_i)_\bbQ\simeq K_0(\cA\otimes_kl)_\bbQ$. Therefore, given an element $\alpha \in K_0(\cA\otimes_k l)_\bbQ$, there exists an index $i_o \in I$ and an element $\alpha_{i_o} \in K_0(\cA\otimes_k k_{i_o})_\bbQ$ such that $\alpha_{i_o} \otimes^{\bf L}_{k_{i_o}} l = \alpha$. Since $k$ is algebraically closed and the $k$-algebra $k_{i_o}$ is finitely generated, the Hilbert's nullstellensatz theorem implies that the $k$-scheme $\mathrm{Spec}(k_{i_o})$ admits a rational point $p\colon \mathrm{Spec}(k) \to \mathrm{Spec}(k_{i_o})$. Hence, we can consider the element $\alpha_{i_o} \otimes^{\bf L}_{k_{i_o}} k \in K_0(\cA)_\bbQ$. We now claim the following:
\begin{eqnarray}\label{eq:claim}
\chi(\alpha, \beta)= \chi((\alpha_{i_o} \otimes^{\bf L}_{k_{i_o}} k)\otimes_k l, \beta) && \forall \,\beta \in K_0(\cA\otimes_k l)_\bbQ\,.
\end{eqnarray}
Note that since $\alpha$ is arbitrary, this claim would imply the surjectivity of \eqref{eq:iso-2}. As above, given an element $\beta \in K_0(\cA\otimes_k l)_\bbQ$, there exists an index $i'_o \in I$ and an element $\beta_{i'_o} \in K_0(\cA\otimes_k k_{i'_o})_\bbQ$ such that $\beta_{i'_o} \otimes^{\bf L}_{k_{i'_o}} l = \beta$. Since $I$ is a filtered diagram, we can (and will) assume without loss of generality that there exists a morphism $i_o \to i'_o$ in $I$. In particular, this yields the base-change homomorphism $-\otimes^{\bf L}_{k_{i_o}} k_{i'_o}\colon K_0(\cA\otimes_k k_{i_o})_\bbQ \to K_0(\cA\otimes_k k_{i'_o})_\bbQ$. Therefore, thanks to the general Remark \ref{rk:generalization1}, in order to prove the above claim \eqref{eq:claim}, it suffices to show that 
\begin{equation}\label{eq:claim2}
\chi(\alpha_{i_o} \otimes^{\bf L}_{k_{i_o}} k_{i'_o}, \beta_{i'_o}) = \chi((\alpha_{i_o} \otimes^{\bf L}_{k_{i_o}} k)\otimes_k k_{i'_o}, \beta_{i'_o})\,.
\end{equation}
Since $k$ is algebraically closed and the $k$-algebra $k_{i'_o}$ is finitely generated, the $k$-scheme $\mathrm{Spec}(k_{i'_o})$ admits a rational point $q\colon \mathrm{Spec}(k) \to \mathrm{Spec}(k_{i'_o})$. Hence, we can consider the base-change homomorphism $-\otimes^{\bf L}_{k_{i'_o}} k\colon K_0(\cA\otimes_k k_{i'_o})_\bbQ \to K_0(\cA)_\bbQ$. Making use once again of the general Remark \ref{rk:generalization1}, we observe that \eqref{eq:claim2} holds if and only if the following equality holds:
\begin{equation*}\label{eq:claim3}
\chi((\alpha_{i_o} \otimes^{\bf L}_{k_{i_o}} k_{i'_o})\otimes^{\bf L}_{k_{i'_o}} k, \beta_{i'_o} \otimes^{\bf L}_{k_{i'_o}} k)= \chi(((\alpha_{i_o} \otimes^{\bf L}_{k_{i_o}} k)\otimes_k k_{i'_o})\otimes^{\bf L}_{k_{i'_o}} k, \beta_{i'_o} \otimes^{\bf L}_{k_{i'_o}} k)\,.
\end{equation*}
Thanks to the general Remark \ref{rk:generalization1}, the left-hand side is equal to
\begin{eqnarray}
 & = &  \chi(\alpha_{i_o} \otimes^{\bf L}_{k_{i_o}} k_{i'_o}, (\beta_{i'_o} \otimes^{\bf L}_{k_{i'_o}}k)\otimes_k k_{i'_o}) \label{eq:star-1111}\\
& = & \chi(\alpha_{i_o}, (\beta_{i'_o} \otimes^{\bf L}_{k_{i'_o}} k)\otimes_k k_{i_o}) \nonumber \\
& = & \chi(\alpha_{i_o} \otimes^{\bf L}_{k_{i_o}} k, \beta_{i'_o} \otimes^{\bf L}_{k_{i'_o}} k) \label{eq:star-2222}\,,
\end{eqnarray}
where in \eqref{eq:star-1111}, resp. \eqref{eq:star-2222}, we are (implicitly) using the fact that the composition $k \to k_{i'_o} \stackrel{q}{\to} k$, resp. $k \to k_{i_o} \stackrel{p}{\to} k$, is equal to the identity. Similarly, thanks to the general Remark \ref{rk:generalization1}, the right-hand side is equal to
\begin{eqnarray}
& = & \chi((\alpha_{i_o} \otimes^{\bf L}_{k_{i_o}} k)\otimes_k k_{i'_o}, (\beta_{i'_o} \otimes^{\bf L}_{k_{i'_o}} k) \otimes_k k_{i'_o}) \label{eq:star-3333} \\
& = &  \chi(\alpha_{i_o} \otimes^{\bf L}_{k_{i_o}} k, \beta_{i'_o} \otimes^{\bf L}_{k_{i'_o}} k) \nonumber\,,
\end{eqnarray}
where in \eqref{eq:star-3333} we are (implicitly) using the fact that the composition $k \to k_{i'_o} \stackrel{q}{\to} k$ is equal to the identity. This proves the above claim \eqref{eq:claim} and hence shows that the base-change homomorphism \eqref{eq:iso-2} is invertible. Finally, note that the above proof also holds integrally. Consequently, base-change induces an isomorphism:
$$ - \otimes_k l \colon K_0(\cA)/_{\!\!\sim \mathrm{num}} \stackrel{\simeq}{\too} K_0(\cA\otimes_k l)/_{\!\!\sim \mathrm{num}}\,.$$
%On the one hand, by combining the (general) commutative diagram \eqref{eq:square-key} with the fact that the composition $k \to k_{i'_o} \stackrel{q}{\to} k$ is equal to the identity, we observe that the left-hand side is equal to:
%$$
% \chi(\alpha_{i_o} \otimes^{\bf L}_{k_{i_o}} k_{i'_o}, (\beta_{i'_o} \otimes^{\bf L}_{k_{i'_o}}k)\otimes_k k_{i'_o}) \label{eq:star1}\\
% =  \chi(\alpha_{i_o}, (\beta_{i'_o} \otimes^{\bf L}_{k_{i'_o}} k)\otimes_k k_{i_o}) \label{eq:star2}\,.
%$$
%On the other hand, the combining the (general) commutative diagram \eqref{eq:square-key} with the fact that the compositions $k \to k_{i'_o} \stackrel{q}{\to} k$ and $k \to k_{i_o} \stackrel{p}{\to} k$ are equal to the identity, we observe that the right-hand side is equal to:
%$$
%\chi(\alpha_{i_o} \otimes^{\bf L}_{k_{i_o}} k, \beta_{i'_o} \otimes^{\bf L}_{k_{i'_o}} k) =  \chi(\alpha_{i_o}, (\beta_{i'_o} \otimes^{\bf L}_{k_{i'_o}} k)\otimes_k k_{i_o}) \label{eq:star22}\,.
%$$
%Finally, note that in the above proof the $\bbQ$-coefficients are used solely in Propositions \ref{prop:Galois} and \ref{prop:inseparable}. \Goncalo{Whenever $l/k$ is a field extension with $k$ algebraically closed}, the Propositions \ref{prop:Galois} and \ref{prop:inseparable} are not necessary. Therefore, in this latter case a proof similar to the above one shows that base-change induces an isomorphism

Let us now assume that $l/k$ is a field extension of separably closed fields. Consider the associated field extension $l_{\mathrm{alg}}/k_{\mathrm{alg}}$, where $l_{\mathrm{alg}}$ stands for ``the'' algebraic closure of $l$ and $k_{\mathrm{alg}}$ for the algebraic closure of $k$ inside $l_{\mathrm{alg}}$. Since by assumption $k$ and $l$ are separably closed, both field extensions $l_{\mathrm{alg}}/l$ and $k_{\mathrm{alg}}/k$ are purely inseparable. Under these notations, we have the following diagram:
$$
\xymatrix@C=1.5em@R=2.5em{
K_0(\cA\otimes_k l)_\bbQ \ar[rr]^-{-\otimes_l l_{\mathrm{alg}}} \ar@{->>}[d] & & K_0(\cA\otimes_k l_{\mathrm{alg}})_\bbQ \ar[d] \\
K_0(\cA\otimes_k l)_\bbQ/_{\!\!\sim \mathrm{num}} \ar[rr]^-{-\otimes_l l_{\mathrm{alg}}} && K_0(\cA\otimes_k l_{\mathrm{alg}})_\bbQ/_{\!\!\sim \mathrm{num}} \\
K_0(\cA)_\bbQ/_{\!\!\sim \mathrm{num}} \ar[rr]^-{-\otimes_k k_{\mathrm{alg}}} \ar@{..>}[u] && K_0(\cA\otimes_k k_{\mathrm{alg}})_\bbQ/_{\!\!\sim \mathrm{num}} \ar[u]_-{-\otimes_{k_{\mathrm{alg}}}l_{\mathrm{alg}}} \\
K_0(\cA)_\bbQ \ar@{->>}[u] \ar[rr]^-{-\otimes_k k_{\mathrm{alg}}} \ar@/^5pc/[uuu]^-{-\otimes_kl} && K_0(\cA\otimes_k k_{\mathrm{alg}})_\bbQ \ar[u] \ar@/_6pc/[uuu]_-{-\otimes_{k_{\mathrm{alg}}}l_{\mathrm{alg}}}  \,.
}
$$
Thanks to Proposition \ref{prop:inseparable} and to the above considerations, the three ``solid'' base-change homomorphisms in the central square are invertible. This implies that there exists a unique ``dashed'' isomorphism making the central square commute. The above diagram implies moreover that the latter ``dashed'' isomorphism is induced by base-change $-\otimes_k l$. Hence, the proof is finished.

Finally, let $l/k$ be a primary field extension. Thanks to Proposition \ref{prop:inseparable}, we can assume without loss of generality that $l/k$ is regular. Consider the associated field extension $l_{\mathrm{sep}}/k_{\mathrm{sep}}$, where $l_{\mathrm{sep}}$ stands for ``the'' separable closure of $l$ and $k_{\mathrm{sep}}$ for the separable closure of $k$ inside $l_{\mathrm{sep}}$. The field extensions $l_{\mathrm{alg}}/l$ and $k_{\mathrm{alg}}/k$ are Galois. Moreover, since $l/k$ is regular, the homomorphism $\mathrm{Gal}(l_{\mathrm{sep}}/l) \twoheadrightarrow \mathrm{Gal}(k_{\mathrm{sep}}/k)$ is surjective. Under these notations, we have the following  diagram:
$$
\xymatrix@C=1.5em@R=2.5em{
K_0(\cA\otimes_k l)_\bbQ \ar[rr]^-{-\otimes_l l_{\mathrm{sep}}} \ar@{->>}[d] & & K_0(\cA\otimes_k l_{\mathrm{alg}})_\bbQ^{\mathrm{Gal}(l_{\mathrm{sep}}/l)} \ar[d] \\
K_0(\cA\otimes_k l)_\bbQ/_{\!\!\sim \mathrm{num}} \ar[rr]^-{-\otimes_l l_{\mathrm{sep}}} && (K_0(\cA\otimes_k l_{\mathrm{sep}})_\bbQ/_{\!\!\sim \mathrm{num}})^{\mathrm{Gal}(l_{\mathrm{sep}}/l)} \\
K_0(\cA)_\bbQ/_{\!\!\sim \mathrm{num}} \ar[rr]^-{-\otimes_k k_{\mathrm{sep}}} \ar@{..>}[u] && (K_0(\cA\otimes_k k_{\mathrm{sep}})_\bbQ/_{\!\!\sim \mathrm{num}})^{\mathrm{Gal}(k_{\mathrm{sep}}/k)} \ar[u]_-{-\otimes_{k_{\mathrm{sep}}}l_{\mathrm{sep}}} \\
K_0(\cA)_\bbQ \ar@{->>}[u] \ar[rr]^-{-\otimes_k k_{\mathrm{sep}}} \ar@/^5pc/[uuu]^-{-\otimes_kl} && K_0(\cA\otimes_k k_{\mathrm{sep}})_\bbQ^{\mathrm{Gal}(k_{\mathrm{sep}}/k)} \ar[u] \ar@/_8pc/[uuu]_-{-\otimes_{k_{\mathrm{sep}}}l_{\mathrm{sep}}}  \,.
}
$$
Thanks to Proposition \ref{prop:Galois} and to the above considerations, the three ``solid'' base-change homomorphisms in the central square are invertible. This implies that there exists a unique ``dashed'' isomorphism making the central square commute. The above diagram implies moreover that the latter ``dashed'' isomorphism is induced by base-change $-\otimes_k l$. Hence, the proof is finished.
%
%where $l_{\mathrm{sep}}$ stands for ``the'' separable closure of $l$ and $k_{\mathrm{sep}}$ for the separable closure of $k$ inside $l_{\mathrm{sep}}$. Proposition \ref{prop:Galois} implies that both horizontal homomorphisms in \eqref{eq:square-last} are invertible. Moreover, since the field extension $l/k$ is regular, the induced homomorphism $\mathrm{Gal}(l_{\mathrm{sep}}/l) \twoheadrightarrow \mathrm{Gal}(k_{\mathrm{sep}}/k)$ is surjective. Therefore, in order to prove that \eqref{eq:iso-2} is invertible it suffices to address the case where $l/k$ is a field extension of separably closed fields. Given such a field extension $l/k$, consider the associated field extension $l_{\mathrm{alg}}/k_{\mathrm{alg}}$, where $l_{\mathrm{alg}}$ stands for ``the'' algebraic closure of $l$ and $k_{\mathrm{alg}}$ for the algebraic closure of $k$ inside $l_{\mathrm{alg}}$. Since $k$ and $l$ are separably closed, both  field extensions $l_{\mathrm{alg}}/l$ and $k_{\mathrm{alg}}/k$ are purely inseparable. Therefore, thanks once again to Proposition \ref{prop:inseparable}, in order to prove that \eqref{eq:iso-2} is invertible it suffices to address the particular case where $l/k$ is a field extension of algebraically closed fields. \Goncalo{In what follows, we prove that \eqref{eq:iso-2} is invertible in the case where $l/k$ is a field extension with $k$ algebraically closed; note that this proves the first claim of item (i).} 

%-------------------------------------------------------------------------------
\subsection*{Proof of item (ii)}
%-------------------------------------------------------------------------------
We start by describing the behavior of mod-$n$ algebraic $K$-theory with respect to purely inseparable field extensions.
\begin{proposition}[Purely inseparable]\label{prop:inseparable2}
Given a purely inseparable field extension $l/k$, a dg category $\cA$, and an integer $n$ coprime to $\mathrm{char}(k)$, we have an isomorphism:
\begin{equation}\label{eq:induced-last1}
-\otimes_k l \colon K_\ast(\cA;\bbZ/n) \stackrel{\simeq}{\too} K_\ast(\cA\otimes_k l; \bbZ/n)\,.
\end{equation}
\end{proposition}
\begin{proof}
Let us assume first that the field extension $l/k$ is finite of degree $d$. As in the proof of Proposition \ref{prop:Galois}, we have the following two homomorphisms:
\begin{eqnarray*}
K_\ast(\cA;\bbZ/n) \stackrel{-\otimes_k l}{\too} K_\ast(\cA\otimes_k l; \bbZ/n) && K_\ast(\cA\otimes_k l; \bbZ/n) \stackrel{\mathrm{res}}{\too} K_\ast(\cA;\bbZ/n)\,.
\end{eqnarray*}
Clearly, the composition $\mathrm{res} \circ (-\otimes_k l)$ is equal to $d\cdot \id$. Similarly to the proof of Proposition \ref{prop:inseparable}, the converse composition $\mathrm{res}(-)\otimes_k l$ is also equal to $d\cdot \id$. Since the field extension $l/k$ is purely inseparable, the degree $d$ is a power of $\mathrm{char}(k)$. Therefore, making use of the fact that the (graded) abelian groups $K_\ast(\cA; \bbZ/n)$ and $K_\ast(\cA\otimes_k l; \bbZ/n)$ are $\bbZ/n^2$-modules and that $n$ is coprime to $\mathrm{char}(k)$, we conclude that the above induced homomorphism \eqref{eq:induced-last1} is invertible.

Let us now assume that the field extension $l/k$ is infinite. In this case, $l$ identifies with the colimit of the filtrant diagram $\{l_i \}_{i \in I}$ of all those intermediate field extensions $l/l_i/k$ which are finite and purely inseparable over $k$. Since $\mathrm{colim}_{i \in I} l_i \simeq l$, we have $\mathrm{colim}_{i \in I} \cA\otimes_k l_i\simeq \cA\otimes_k l$. Using the fact that the functor \eqref{eq:K-coefficients} preserves filtered colimits, we hence conclude that $\mathrm{colim}_{i \in I} K_\ast(\cA\otimes_k l_i; \bbZ/n)\simeq K_\ast(\cA\otimes_k l; \bbZ/n)$. Consequently, the proof follows from the preceding finite dimensional case.
\end{proof}
Consider the field extension $l_{\mathrm{alg}}/k_{\mathrm{alg}}$, where $l_{\mathrm{alg}}$ stands for ``the'' algebraic closure of $l$ and $k_{\mathrm{alg}}$ for the algebraic closure of $k$ inside $l_{\mathrm{alg}}$. Since $k$ and $l$ are separably closed, the extensions $l_{\mathrm{alg}}/l$ and $k_{\mathrm{alg}}/k$ are purely inseparable. Therefore, thanks to Proposition \ref{prop:inseparable2}, in order to prove that \eqref{eq:iso-3} is invertible it suffices to address the particular case where $l/k$ is a field extension of algebraically closed fields. 

We start by proving that \eqref{eq:iso-3} is injective. Note that $l$ identifies with the colimit of the filtrant diagram $\{k_i\}_{i \in I}$ of all those finitely generated $k$-algebras $k_i$ which are contained in $l$. Without loss of generality, we may assume that $k_i$ is integrally closed in its field of fractions. Hence, each such $k$-algebra $k_i$ corresponds to an irreducible smooth affine $k$-curve $\mathrm{Spec}(k_i)$. Since $\mathrm{colim}_{i \in I} k_i \simeq l$, we have $\mathrm{colim}_{i \in I} \cA\otimes_k k_i \simeq \cA\otimes_k l$. Using the fact that the functor \eqref{eq:K-coefficients} preserves filtered (homotopy) colimits, we hence obtain the following isomorphism: 
\begin{equation}\label{eq:colim}
\mathrm{colim}_{i \in I} K_\ast(\cA\otimes_k k_i; \bbZ/n) \simeq K_\ast(\cA\otimes_k l; \bbZ/n)\,.
\end{equation}
By assumption, the field $k$ is algebraically closed. Therefore, the Hilbert's nullstellensatz theorem implies that all such smooth affine curves $\mathrm{Spec}(k_i)$ admit a rational point $p_i\colon \mathrm{Spec}(k) \to \mathrm{Spec}(k_i)$. Consequently, the following compositions
\begin{eqnarray*}
K_\ast(\cA;\bbZ/n) \stackrel{-\otimes_k k_i}{\too} K_\ast(\cA\otimes_k k_i; \bbZ/n) \stackrel{-\otimes_{k_i}^{\bf L}k}{\too} K_\ast(\cA;\bbZ/n) && i \in I
\end{eqnarray*}
are equal to the identity. This shows, in particular, that the homomorphisms $-\otimes_k k_i$ are injective. Making use of the above isomorphism \eqref{eq:colim}, we hence conclude that the homomorphism \eqref{eq:iso-3} is also injective.

We now prove that \eqref{eq:iso-3} is surjective. Let $\alpha$ be an element of $K_\ast(\cA\otimes_k l;\bbZ/n)$. Thanks to the above isomorphism \eqref{eq:colim}, there exists an index $i_o \in I$ and an element $\alpha_{i_o}$ of $K_\ast(\cA\otimes_k k_i;\bbZ/n)$ such that $\alpha_{i_o}\otimes^{\bf L}_{k_{i_o}} l=\alpha$. Choose a rational point $p\colon \mathrm{Spec}(k) \to \mathrm{Spec}(k_{i_o})$ of the smooth affine $k$-curve $\mathrm{Spec}(k_{i_o})$ and consider the following commutative diagram:
\begin{equation}\label{eq:22-squares}
\xymatrix{ 
 l & k_{i_o}\otimes_k l \ar[l]_-{\overline{p}} \ar[r] & l \\
 k \ar[u] & k_{i_o} \ar[l]_-p \ar[u] \ar[r] & l \ar@{=}[u]
}
\end{equation}
The two upper horizontal maps in \eqref{eq:22-squares} may be understood as two rational points of the smooth affine $l$-curve $\mathrm{Spec}(k_{i_o}\otimes_k l)$. Hence, thanks to Proposition \ref{prop:rigidity} below, the associated homomorphisms from $K_\ast(\cA\otimes_k (k_{i_o}\otimes_k l); \bbZ/n)$ to $K_\ast(\cA\otimes_k l ;\bbZ/n)$ agree. This implies that $(\alpha_{i_o}\otimes^{\bf L}_{k_{i_o}}k)\otimes_k l = \alpha_{i_o}\otimes^{\bf L}_{k_{i_o}} l = \alpha$. Since $\alpha_{i_o}\otimes^{\bf L}_{k_{i_o}}k$ belongs to $K_\ast(\cA;\bbZ/n)$, we then conclude that the homomorphism \eqref{eq:iso-3} is surjective.

\begin{proposition}\label{prop:rigidity}
Let $C=\mathrm{Spec}(R)$ be a smooth affine $l$-curve. Given any two rational points $p, q \colon \mathrm{Spec}(l) \to C$, the associated homomorphisms $\id \otimes p^\ast = -\otimes^{\bf L}_R l$ and $\id \otimes q^\ast = -\otimes_R^{\bf L} l$ from $K_\ast(\cA\otimes_k R; \bbZ/n)$ to $K_\ast(\cA\otimes_k l ; \bbZ/n)$ agree.
\end{proposition}
\begin{proof}
Consider the following homomorphism 
\begin{eqnarray*}
\mathrm{Div}(C) \stackrel{\theta}{\too} \Hom(K_\ast(\cA\otimes_k R; \bbZ/n),K_\ast(\cA\otimes_k l; \bbZ/n)) && p \mapsto (\id \otimes p^\ast)\,,
\end{eqnarray*}
where $\mathrm{Div}(C)$ stands for the abelian group of divisors on $C$. Choose a smooth compactification\footnote{Recall that the smooth compactification $\overline{C}$ is unique up to isomorphism.} $\overline{C}$ of $C$, with closed complement $C_\infty$ consisting of a finite set of points, and consider the relative Picard group $\mathrm{Pic}(\overline{C},C_\infty)$. Recall that $\mathrm{Pic}(\overline{C}, C_\infty)$ is defined as the quotient of $\mathrm{Div}(C)$ by the following equivalence relation: $D\sim D'$ if there exists a rational function $f\colon \overline{C} \to \bbP^1$ such that $f_{|C_\infty}=1$, $f^{-1}(0)=D$ and $f^{-1}(\infty)=D'$. Thanks to Lemma \ref{lem:key} below, $\theta$ factors through $\mathrm{Pic}(\overline{C},C_\infty)$. Since the (graded) abelian group $\Hom(K_\ast(\cA\otimes_k R; \bbZ/n),K_\ast(\cA\otimes_k l; \bbZ/n))$ is $n^2$-torsion, the homomorphism $\theta$ factors moreover through the quotient $\mathrm{Pic}(\overline{C},C_\infty)/n^2$. Now, using the fact that the kernel $\mathrm{Pic}^0(\overline{C}, C_\infty)$ of the degree map $\mathrm{deg}\colon \mathrm{Pic}(\overline{C},C_\infty) \to \bbZ$ is a $n$-divisible group (since it agrees with the group of $k$-points of the Rosenlicht Jacobian of $\overline{C}$) and that the difference $p-q$ belongs to $\mathrm{Pic}^0(\overline{C},C_\infty)$, we hence conclude that $(\id \otimes p^\ast)=(\id \otimes q^\ast)$.
\end{proof}
\begin{lemma}\label{lem:key}
The above homomorphism $\theta$ factors through $\mathrm{Pic}(\overline{C}, C_\infty)$.
\end{lemma}
\begin{proof}
Since the field $l$ is algebraically closed, the relative Picard group $\mathrm{Pic}(\overline{C},C_\infty)$ is generated by the unramified divisors on $C$. Hence, it suffices to show that the homomorphism $\theta$ vanishes on the principal divisors $\mathrm{div}(f):=f^{-1}(0)-f^{-1}(\infty)$ associated to those rational functions $f\colon \overline{C} \to \bbP^1$, with $f_{|C_\infty}=1$, which are unramified over $0$ and $\infty$. Let $D_0:=f^{-1}(0)$, $D_\infty:=f^{-1}(\infty)$, and $\cU:=f^{-1}(\bbP^1\backslash \{1\})$. By definition, $\cU:=\mathrm{Spec}(S)$ is an affine open subscheme of $C$ which contains $D_0$ and $D_\infty$. Moreover, $f$ restricts to a finite flat map $f\colon \cU \to \bbP^1\backslash \{1\}=\bbA^1$. Hence, without loss of generality, we can replace $K_\ast(\cA\otimes_k R; \bbZ/n)$ by $K_\ast(\cA\otimes_k S; \bbZ/n)$.

Consider the following commutative diagrams:
$$
\xymatrix{
D_0=f^{-1}(0) \ar[r]^-{j_0} \ar[d]_-{f_0} & \cU \ar[d]^-f && D_\infty= f^{-1}(\infty) \ar[d]_-{f_\infty} \ar[r]^-{j_\infty} & \cU \ar[d]^-f \\
\mathrm{Spec}(l) \ar[r]^-{i_0} & \bbA^1 && \mathrm{Spec}(l) \ar[r]^-{i_\infty} & \bbA^1\,.
}
$$ 
Thanks to flat proper base-change, they yield commutative diagrams
$$
\xymatrix{
\perf_\dg(D_0) \ar[d]_-{(f_0)_\ast}  & \perf_\dg(\cU) \ar[l]_-{j_0^\ast} \ar[d]^-{f_\ast} & \perf_\dg(D_\infty) \ar[d]_-{(f_\infty)_\ast}  & \perf_\dg(\cU) \ar[l]_-{j_\infty^\ast} \ar[d]^-{f_\ast} \\
\perf_\dg(\mathrm{Spec}(l)) & \perf_\dg(\bbA^1) \ar[l]_-{i_0^\ast}& \perf_\dg(\mathrm{Spec}(l))  & \perf_\dg(\bbA^1) \ar[l]_-{i_\infty^\ast}
}
$$
in the homotopy category $\Hmo(k)$. Since $f$ is unramified over $0$ and $\infty$ and $D_0 = \amalg_{f^{-1}(0)}\mathrm{Spec}(l)$ and $D_\infty=\amalg_{f^{-1}(\infty)}\mathrm{Spec}(l)$, the push-forward dg functors $(f_0)_\ast$ and $(f_\infty)_\ast$ corresponds to the fold identity dg functors from $\amalg_{f^{-1}(0)} \perf_\dg(\mathrm{Spec}(l))$ and $\amalg_{f^{-1}(\infty)} \perf_\dg(\mathrm{Spec}(l))$ to $\perf_\dg(\mathrm{Spec}(l))$, respectively. Consequently, making use of the Morita equivalences $\perf_\dg(\mathrm{Spec}(l))\simeq l$ and $\perf_\dg(\cU)\simeq S$, by applying the functor $K_\ast(\cA\otimes_k -; \bbZ/n)$ to the preceding commutative diagrams, we conclude that
\begin{eqnarray*}
\theta(D_0) & =& (\id \otimes (f_0)_\ast)\circ (\id \otimes j_0^\ast) = (\id \otimes i_0^\ast) \circ (\id \otimes f_\ast)\\
\theta(D_\infty) & =&  (\id \otimes (f_\infty)_\ast)\circ (\id \otimes j_\infty^\ast) = (\id \otimes i_\infty^\ast) \circ (\id \otimes f_\ast)\,.
\end{eqnarray*}
As proved in \cite[Thm.~1.2(i)]{Klein}, since $n$ is coprime to $\mathrm{char}(k)$, the functor \eqref{eq:K-coefficients} is $\bbA^1$-homotopy invariant. Hence, the equality $\id \otimes i_0^\ast = \id \otimes i_\infty^\ast$ holds. This allows us to conclude that $\theta(\mathrm{div}(f))=\theta(D_0 - D_\infty)=0$, and so the proof is finished.
\end{proof}
%
%where $\overline{x}$ and $\overline{\varrho_i}$ are the induced maps. Note that these latter two maps may be understood as two rational points of the smooth affine curve $\mathrm{Spec}(R_i\otimes l)$ defined over $l$. Therefore, making use Corollary \ref{cor:key} (with $k$ replaced by $l$, $x$ replaced by $\overline{x}$ and $y$ replaced by $\overline{\varrho_i}$), we conclude that the image of the element $(\id \otimes x_\ast)(\beta_i)\in K_\ast(\cA;\bbZ/n)$ under the induced homomorphism \eqref{eq:induced2} also agrees with $\beta$. This implies that \eqref{eq:induced2} is moreover surjective.
%\begin{remark}
%Note that the above proof of the injectivity of \eqref{eq:induced2} does not uses the assumption that $l$ is algebraically closed.
%\end{remark}

Finally, we prove that the isomorphism \eqref{eq:iso-3} also holds for the variants $K_\ast(-)^\wedge_n$, $K^{\mathrm{et}}_\ast(-;\bbZ/n)$, $K_\ast(-)/n$ and ${}_n K_\ast(-)$. The cases of $n$-adic algebraic $K$-theory and mod-$n$ {\'e}tale $K$-theory follow automatically from their definition. In what concerns the other two variants, note that the above proof of the injectivity of \eqref{eq:iso-3} holds {\em mutatis mutandis} with mod-$n$ algebraic $K$-theory replaced by nonconnective algebraic $K$-theory. This implies, in particular, that the base-change homomorphisms
\begin{eqnarray}\label{eq:induced-last}
& -\otimes_k l \colon K_\ast(\cA)/n \too K_\ast(\cA \otimes_k l)/n & -\otimes_k l \colon {}_nK_\ast(\cA) \too {}_n K_\ast(\cA\otimes_k l)
\end{eqnarray}
are injective. Making use of the following commutative diagrams
$$
\xymatrix{
0 \ar[r] & K_\ast(\cA\otimes_k l)/n \ar[r] & K_\ast(\cA\otimes_k l; \bbZ/n) \ar[r] & {}_n K_{\ast-1}(\cA\otimes_k l) \ar[r] & 0 \\
0 \ar[r] & K_\ast(\cA)/n \ar[u]^-{-\otimes_k l} \ar[r] & K_\ast(\cA; \bbZ/n) \ar[u]^-{-\otimes_k l}_-\simeq \ar[r] & {}_n K_{\ast-1}(\cA) \ar[r] \ar[u]^-{-\otimes_k l}& 0
}
$$
and of the snake lemma, we hence conclude that the base-change homomorphisms \eqref{eq:induced-last} are moreover surjective. This concludes the proof of Theorem \ref{thm:main}.
%-------------------------------------------------------------------------------
\section{$n$-adic noncommutative mixed motives}\label{sec:adic}
%-------------------------------------------------------------------------------
Recall from \cite[\S8.3]{book} the construction of the closed symmetric monoidal Quillen model category $\mathrm{NMot}(k):=L_{\mathrm{loc}}\Fun(\dgcat_{\mathrm{f}}(k)^\op, \Spt)$ and of the associated symmetric monoidal functor $\mathrm{U}\colon \dgcat(k) \to \NMot(k)$. The triangulated category of noncommutative mixed motives $\mathrm{NMix}(k)$ is defined as the smallest thick triangulated subcategory of $\Ho(\mathrm{NMot}(k))$ containing the objects $\mathrm{U}(\cA)$ with $\cA$ a smooth proper dg category; see \cite[\S9.1]{book}. By construction, the category $\Ho(\NMot(k))$, and hence $\mathrm{NMix}(k)$, is enriched over $\Ho(\Spt)$.

Let $L_{\bbS/n}(\Spt)$ be the closed symmetric monoidal Quillen model category of $\bbS/n$-local symmetric spectra, $\NMot(k)^\wedge_n$ the closed symmetric monoidal Quillen model category $L_{\mathrm{loc}}\Fun(\dgcat_{\mathrm{f}}(k)^\op, L_{\bbS/n}(\Spt))$, and $\mathrm{U}(-)^\wedge_n\colon \dgcat(k) \to \NMot(k)^\wedge_n$ the associated symmetric monoidal functor. The triangulated category of {\em $n$-adic noncommutative mixed motives $\mathrm{NMix}(k)^\wedge_n$} is defined as the smallest thick triangulated subcategory of $\Ho(\NMot(k)^\wedge_n)$ containing the objects $\mathrm{U}(\cA)^\wedge_n$ with $\cA$ a smooth proper dg category. As proved in \cite[Thm.~1.43]{book}, the smooth proper dg categories can be characterized as the strongly dualizable objects of the symmetric monoidal category $\Hmo(k)$; the dual of a smooth proper dg category $\cA$ is given by the opposite dg category $\cA^\op$. Since the functor $\mathrm{U}(-)^\wedge_n$ is symmetric monoidal, we hence conclude that the objects $\mathrm{U}(\cA)^\wedge_n$, with $\cA$ smooth proper, are strongly~dualizable and consequently that the symmetric monoidal category $\mathrm{NMix}(k)^\wedge_n$ is rigid.
%The above adjunction \eqref{eq:adjunction1} gives naturally rise to the following adjunction
%\begin{equation}\label{eq:adjunction11}
%\xymatrix{
%\NMot(k)^\wedge_n \ar@<1ex>[d]^-{\mathrm{forget}}\\
%\NMot(k) \ar@<1ex>[u]^-{(-)^\wedge_n}\,,
%}
%\end{equation}
\begin{proposition}
Given dg categories $\cA$ and $\cB$, with $\cA$ smooth and proper, we have a natural isomorphism of (symmetric) spectra:
\begin{equation}\label{eq:representability1}
{\bf R}\Hom(\mathrm{U}(\cA)^\wedge_n, \mathrm{U}(\cB)^\wedge_n) \simeq K(\cA^\op \otimes \cB)^\wedge_n\,.
\end{equation}
\end{proposition}
\begin{proof}
Since $\cA$ is smooth and proper, the object $\mathrm{U}(\cA)^\wedge_n$ is strongly dualizable with dual $\mathrm{U}(\cA^\op)^\wedge_n$. Hence, we have the following isomorphisms
\begin{eqnarray}
{\bf R}\Hom(\mathrm{U}(\cA)^\wedge_n, \mathrm{U}(\cB)^\wedge_n) & \simeq & {\bf R} \Hom(\mathrm{U}(k)^\wedge_n, \mathrm{U}(\cA^\op)^\wedge_n \otimes \mathrm{U}(\cB)^\wedge_n) \nonumber\\
& \simeq & {\bf R}\Hom(\mathrm{U}(k)^\wedge_n, \mathrm{U}(\cA^\op \otimes B)^\wedge_n)\nonumber\\
& \simeq & {\bf R}\Hom(\mathrm{U}(k), \mathrm{holim}_\nu \mathrm{U}(\cA^\op \otimes \cB)/n^\nu) \label{eq:star-4444}\\
& \simeq & \mathrm{holim}_\nu {\bf R}\Hom(\mathrm{U}(k), \mathrm{U}(\cA^\op \otimes \cB))/n^\nu \nonumber\\
& \simeq & \mathrm{holim}_\nu K(\cA^\op \otimes \cB)/n^\nu =: K(\cA^\op \otimes \cB)^\wedge_n\,, \label{eq:star-5555}
\end{eqnarray}
where \eqref{eq:star-4444} follows from the induced adjunction between $\NMot(k)$ and $\NMot(k)^\wedge_n$ and \eqref{eq:star-5555} from \cite[Thm.~8.28]{book}.
\end{proof}
%-------------------------------------------------------------------------------
\subsection*{Proof of Theorem \ref{thm:NCmotivic}}
%-------------------------------------------------------------------------------
By construction, every object of $\NNum(k)_\bbQ$ is, up to a direct summand, of the form $U(\cA)_\bbQ$; see Remark \ref{rk:false}. Moreover, as explained in \cite[\S4.7]{book}, we have natural isomorphisms of $\bbQ$-vector spaces:
$$\Hom(U(\cA)_\bbQ,U(\cB)_\bbQ) \simeq K_0(\cA^\op \otimes \cB)_\bbQ/_{\!\!\sim \mathrm{num}}\,.$$ Therefore, the proof of item (i) follows from Theorem \ref{thm:main}(i).

By construction, the triangulated category $\mathrm{NMix}(k;\bbZ/n)$ comes equipped with a functor $\mathrm{U}(-;\bbZ/n)\colon \dgcat_{\mathrm{sp}}(k) \to \NNum(k; \bbZ/n)$. Moreover, it is generated by the objects of the form $\mathrm{U}(\cA;\bbZ/n)$. Furthermore, as explained in \cite[Prop.~4.5]{Hopf}, we have natural isomorphisms of (symmetric) spectra
\begin{equation}\label{eq:representability2}
{\bf R} \Hom(\mathrm{U}(\cA;\bbZ/n), \mathrm{U}(\cB;\bbZ/n))\simeq K(\cA^\op \otimes \cB) \wedge \mathrm{H}(\bbZ/n)\,,
\end{equation}
where $\mathrm{H}(\bbZ/n)$ stands for the Eilenberg-MacLane spectrum of $\bbZ/n$. Therefore, the fully-faithfulness of the base-change functor $-\otimes_k l \colon \mathrm{NMix}(k;\bbZ/n) \to \mathrm{NMix}(l;\bbZ/n)$ follows from Theorem \ref{thm:main}(ii). The proof of the fully-faithfulness of the base-change functor $-\otimes_k l \colon \mathrm{NMix}(k)^\wedge_n \to \mathrm{NMix}(l)^\wedge_n$ is similar; simply replace \eqref{eq:representability2} by \eqref{eq:representability1}.
\begin{remark}[Alternative proof]\label{rk:alternative}
The commutative counterpart of Theorem \ref{thm:NCmotivic}(i) (where the category $\NNum(k)_\bbQ$ is replaced by the classical category of numerical motives $\Num(k)_\bbQ$) was established by Kahn in \cite[Prop.~5.5]{Kahn}. Here is an alternative proof: as explained in \cite[\S4.6-4.7 and \S4.10]{book}, we have the commutative~diagram
\begin{equation}\label{eq:diagram-big}
\xymatrix{
\Num(k)_\bbQ \ar[d] \ar[r]^-{-\otimes_k l} & \Num(l)_\bbQ \ar[d] \\
\Num(k)_\bbQ/_{\!\! -\otimes \bbQ(1)} \ar[r]^-{-\otimes_k l} \ar[d]_-\Phi & \Num(l)_\bbQ/_{\!\! -\otimes \bbQ(1)} \ar[d]^-\Phi \\
\NNum(k)_\bbQ \ar[r]^-{-\otimes_k l} & \NNum(l)_\bbQ\,,
}
\end{equation}
where $\Num(k)_\bbQ/_{\!\! -\otimes \bbQ(1)}$ stands for the orbit category with respect to the Tate motive $\bbQ(1)$. Since the functor $\Phi$ is fully-faithful, it follows then from the combination of Theorem \ref{thm:NCmotivic}(i) with the definition of the orbit category that the upper base-change functor $-\otimes_k l$ in \eqref{eq:diagram-big} is also fully-faithful.
\end{remark}

\medbreak\noindent\textbf{Acknowledgments:} The author is grateful to Joseph Ayoub and Ivan Panin for useful discussions, to Oliver R\"ondigs and Paul Arne {\O}stv{\ae}r for references, and to  Charles Vial for comments on a previous version. The author is also thankful to the Mittag-Leffler Institute for its hospitality.

\end{document}

\end{proof}